\RequirePackage{etoolbox}
\csdef{input@path}{{style/}{graphics/}}
\documentclass[ba,linksfromyear,nameyear,dvips,preprint]{imsart}
\usepackage{natbib}
\usepackage[backref=page]{hyperref}
\usepackage{amsthm,amsmath}
\usepackage{graphicx}
\RequirePackage{amssymb}

\pubyear{2015}
\volume{10}
\issue{1}
\firstpage{31}
\lastpage{51}
\doi{10.1214/14-BA886}% Updated by VTEXPTS2LaTeX.exe, 21.01.2015 13:00

\makeatletter

\theoremstyle{definition}
\newtheorem{thm}{Theorem}

\newtheorem{cor}{Corollary}

\newcommand{\mGammax}[2]{
\raise1.5pt\hbox{$\displaystyle \mathop{\Gamma}$}\hspace{-1.5pt}_{#1}^{\,(\mathrm{m})#2}}
\newcommand{\mGammay}[2]{\tilde{\Gamma}_{#1}^{\,(\mathrm{m})#2}}
\newcommand{\oGammax}[2]{\Gamma_{#1}^{(0)#2}}
\newcommand{\oGammay}[2]{\tilde{\Gamma}_{#1}^{(0)#2}}
\newcommand{\oGammaxy}[2]{\mathring{\Gamma}_{#1}^{(0)#2}}
\newcommand{\eGammax}[2]{\Gamma_{#1}^{(\mathrm{e})#2}}
\newcommand{\eGammay}[2]{\tilde{\Gamma}_{#1}^{\,(\mathrm{e})#2}}
\newcommand{\oGammayx}[2]{\mathring{\Gamma}_{#1}^{{\,(0)}#2}}
\newcommand{\Deltaxy}{\mathring{\Delta}}
\newcommand{\tphi}{\tilde{\phi}}
\newcommand{\Tx}[2]{T_{#1}^{#2}}
\newcommand{\Ty}[2]{\widetilde{T}_{#1}^{#2}}
\newcommand{\Qy}[2]{\tilde{Q}_{#1}^{#2}}
\newcommand{\py}{\displaystyle \tilde{p}}
\newcommand{\lx}{\displaystyle l}
\newcommand{\ly}{\displaystyle \tilde{l}}
\newcommand{\gx}[2]{g_{#1}^{#2}}
\newcommand{\gy}[2]{\tilde{g}_{#1}^{#2}}
\newcommand{\gxy}[2]{\mathring{g}_{#1}^{#2}}
\newcommand{\vxe}[2]{v_{#1}^{\mathrm{(e)}#2}}
\newcommand{\vym}[2]{\tilde{v}_{#1}^{\mathrm{(m)}#2}}
\newcommand{\enablax}[2]{\nabla_{#1}^{(\mathrm{e})#2}}
\newcommand{\onablax}[2]{\nabla_{#1}^{(0)#2}}
\newcommand{\mnablax}[2]{\nabla_{#1}^{(\mathrm{m})#2}}
\newcommand{\enablay}[2]{\widetilde{\nabla}_{#1}^{(\mathrm{e}) #2}}
\newcommand{\onablay}[2]{\widetilde{\nabla}_{#1}^{(0) #2}}
\newcommand{\mnablay}[2]{\widetilde{\nabla}_{#1}^{(\mathrm{m}) #2}}
\newcommand{\onablaxy}[2]{\mathring{\nabla}_{#1}^{(0)#2}}
\newcommand{\upi}{u_\pi}
\newcommand{\upiprime}{u_{\pi'}}
\newcommand{\htheta}{\hat{\theta}}
\newcommand{\dd}{\mbox{d}}
\newcommand{\ee}{\mbox{e}}
\newcommand{\oo}{\mathrm{o}}
\newcommand{\E}{{\rm E}}
\newcommand{\piJ}{\pi_\mathrm{J}}
\newcommand{\piP}{\pi_\mathrm{P}}
\makeatother

\begin{document}

\begin{frontmatter}
\title{Asymptotic Properties of Bayesian Predictive Densities When the Distributions of Data and Target Variables are Different}
\runtitle{Asymptotic Properties of Bayesian Predictive Densities}

\begin{aug}
\author[a]{\fnms{Fumiyasu} \snm{Komaki}\ead[label=e1]{komaki@mist.i.u-tokyo.ac.jp}}\thanksref{t1}%,
\runauthor{F.~Komaki}
% \affiliation{}
\address[a]{Department of Mathematical Informatics, Graduate School of
Information Science and Technology, the University of Tokyo, 7-3-1 Hongo,
Bunkyo-ku, Tokyo 113-8656, JAPAN,\\\printead{e1}}

% \thankstext{<label>}{<text>}
\thankstext{t1}{RIKEN Brain Science Institute, 2-1 Hirosawa, Wako City, Saitama 351-0198, JAPAN}
\end{aug}

%% Abstract %%
\begin{abstract}
Bayesian predictive densities
when the observed data $x$ and
the target variable $y$ to be predicted have different distributions
are investigated by using the framework of information geometry.
The performance of predictive densities is evaluated by the Kullback--Leibler divergence.
The parametric models are formulated as Riemannian manifolds.
In the conventional setting in which $x$ and $y$~have the same distribution,
the Fisher--Rao metric and the Jeffreys prior play essential roles.
In the present setting in which $x$ and $y$ have different distributions,
a new metric, which we call the predictive metric, constructed by using the Fisher information matrices of $x$ and $y$,
and the volume element based on the predictive metric play the corresponding roles.
It is shown that Bayesian predictive densities based on priors constructed by using non-constant positive superharmonic functions
with respect to the predictive metric asymptotically
dominate those based on the volume element prior of the predictive metric.
\end{abstract}

%% Keywords %%

% \begin{keyword}[class=MSC]
% \kwd[Primary ]{}
% \kwd[; secondary ]{}
% \end{keyword}

\end{frontmatter}

%% Mainmatter %%

\vspace*{12pt}\section{Introduction}
Suppose that we have independent observations $x(1), x(2), \ldots, x(N)$ from
a probability density $p(x \mid \theta)$ that belongs to a parametric model $\{p(x \mid \theta) \mid \theta \in \Theta\}$,
where $\theta = (\theta^1, \theta^2, \ldots, \theta^d)$ is an unknown $d$-dimensional parameter and $\Theta$ is the parameter space.
The random variable $y$ to be predicted is independently distributed according to
a density $\py(y \mid \theta)$ in a parametric model $\{\py(y \mid \theta) \mid \theta \in \Theta\}$,
possibly different from $\{p(x \mid \theta) \mid \theta \in \Theta\}$,
with the same parameter $\theta$.
The objective is to construct a predictive density $\hat{p}(y;x^N)$ for $y$ by using $x^N := (x(1),\ldots,x(N))$.
The performance of $\hat{p}(y;x)$ is evaluated by the Kullback--Leibler divergence
\[
 D(\py(y \mid \theta), \hat{p}(y;x^N)) := \int \py(y \mid \theta) \log \frac{\py(y \mid \theta)}{\hat{p}(y;x^N)} \dd y\
\]
from the true density $\py(y \mid \theta)$ to the predictive density $\hat{p}(y;x^N)$.
The risk function is given by
\[
 \E \Bigl[ D(\py(y \mid \theta), \hat{p}(y;x^N)) \, \Big| \, \theta \Bigr]
= \iint p(x^N \mid \theta) \py(y \mid \theta) \log \frac{\py(y \mid \theta)}{\hat{p}(y;x^N)} \dd y \dd x^N.
\]

It is widely recognized that
plug-in densities $\py(y \mid \hat{\theta})$ constructed by replacing the unknown parameter $\theta$
by an estimate $\hat{\theta}(x^N)$
may not perform very well
and that Bayesian predictive densities
\[
 \py_\pi(y \mid x^N) :=
\frac{\int \py(y \mid \theta) p(x^N \mid \theta) \pi(\theta) \dd \theta}
{\int p(x^N \mid \theta) \pi(\theta) \dd \theta}
\]
constructed by using a prior $\pi$ perform better than plug-in densities.
If the value of $\theta$ is given, there is no specific meaning of considering the conditional
density of $y$ given $x^N$ since the obvious relation
$p(y \mid x, \theta) = p(y \mid \theta)$ holds.
However, if $\theta$ is unknown, Bayesian predictive densities $p_\pi(y \mid x^N)$
constructed by introducing a prior density $\pi(\theta)$ on the parameter space are useful
to approximate the true density $p(y \mid \theta)$ as discussed in \citet{AD75cambridge} and \citet{G93cambridge}.
In fact, there exists a predictive density whose asymptotic risk is smaller than
that of a plug-in density unless the mean mixture curvature of the model
manifold vanishes, see \citet{K96biometrika} and \citet{H98as} for details.
The choice of $\pi$ becomes important especially when the sample size $N$ is not very large.
Although the Jeffreys prior is a widely known default prior,
it does not perform satisfactorily especially when the unknown parameter is multidimensional as Jeffreys himself
pointed out.

\citet{K01biometrika} constructed a Bayesian predictive density incorporating the advantage of
shrinkage methods for the multivariate normal model.
See also \citet{GLX06as} for useful results for the normal model.

In the conventional setting in which the distributions of $x(i)$, $i=1,\ldots,N$, and $y$ are the same,
asymptotic theory of prediction based on general parametric models has been studied
by using the framework of information geometry, see \citet{K96biometrika}.
In information geometry, a parametric statistical model is regarded as a differentiable manifold,
which we call the model manifold, and the parameter space
is regarded as a coordinate system of the manifold, see \citet{A85springer}.
The Fisher--Rao metric is
a Riemannian metric based on the Fisher information matrix on the model manifold.
The Jeffreys prior $\pi_\mathrm{J}(\theta)$ corresponds to the volume element of the model manifold associated with the Fisher--Rao metric.
When the distributions of $x(i)$, $i=1,\ldots,N$, and $y$ are the same,
the asymptotic difference between the risks of $\py_\pi(y \mid x^N)$ and $\py_\mathrm{J}(y \mid x^N)$ is given by
\begin{align}
N^2 \biggl[ \E \big\{ & D (\py(y \mid \theta), \py_\pi(y \mid x^N) \, \big| \, \theta \big\}
- \E \big\{ D ( \py(y \mid \theta), \py_\mathrm{J}(y \mid x^N) \, \big| \, \theta \big\} \biggr] \notag\\
=& \frac{\displaystyle \Delta \left(\frac{\pi}{\piJ}\right)}{\displaystyle \left( \frac{\pi}{\piJ} \right) }
- \frac{1}{2} \sum_{i=1}^d \sum_{j=1}^d g^{ij} \frac{\displaystyle \partial_i \left(\frac{\pi}{\piJ}\right) \partial_j \left(\frac{\pi}{\piJ}\right)}
{\displaystyle \left(\frac{\pi}{\piJ}\right)^2}
+ \oo(1)
= 2 \frac{\displaystyle \Delta \left(\frac{\pi}{\piJ}\right)^{\frac{1}{2}}}
{\displaystyle \left( \frac{\pi}{\piJ} \right)^{\frac{1}{2}}}
+ \oo(1),
\label{main0}
\end{align}
where
$\partial_i$ denotes $\partial/\partial \theta^i$,
$g_{ij} := \E \{\partial_i \log p(x \mid \theta) \partial_j \log p(x \mid \theta) \mid \theta\}$,
$g^{ij}$ denotes the $(i,j)$-element of the inverse of the $d \times d$ matrix $(g_{ij})$,
and $\Delta$ is the Laplacian, see \cite{K06as}.
The Laplacian $\Delta$ on a Riemannian manifold endowed with a metric $g_{ij}$ is defined by
\begin{align}
\Delta f = |g|^{-1/2} \sum_i \sum_j \partial_i (|g|^{1/2} g^{ij} \partial_j f)
= \sum_i \sum_j \nabla_{i}^{(0)} (g^{ij} \partial_j f),
\label{deflaplacian}
\end{align}
where $|g|$ is the determinant of the $d \times d$ matrix $(g_{ij})$,
$f$ is a smooth real function on $\Theta$, and $\nabla_{i}^{(0)}$ denotes
the covariant derivative, defined in the next section.
The indices $i,j,k \ldots$ run from $1$ to $d$.
Note that both the definition \eqref{deflaplacian} of the Laplacian
and the definition $\Delta f = -|g|^{-1/2} \sum_i \sum_j \partial_i (|g|^{1/2} g^{ij} \partial_j f)$ that differs in sign
are widely adopted in the mathematics literature,
although it is confusing.
Because of \eqref{main0}, if there exists a non-constant positive superharmonic function $f$,
i.e.\ a non-constant positive function satisfying $\Delta f \leq 0$ for every $\theta$, on the model manifold, then
the Bayesian predictive density based on the prior density defined by $\pi = f \piJ$
asymptotically dominates that based on the Jeffreys prior.
Here, the Riemannian geometric structure of the model manifold based on the Fisher--Rao metric
plays a fundamental role.

In practical applications, it often occurs that observed data $x(i)$, $i = 1,\ldots,N$, and the target variable $y$ to be predicted
have different distributions.
Regression models are a typical example.
Suppose that we observe
$x = W \theta + \varepsilon$, where $W$ is a given $n \times d$ matrix $(n \ge d)$,
and predict $y = Z \theta + \varepsilon$,
where $Z$ is a given $m \times d $ matrix
and $ \theta = (\theta^1 , \dotsb, \theta^d) $ is an unknown parameter.
Then, the Fisher information matrices for the same parameter $\theta$
based on $ p(x \mid \theta)$ and $\tilde{p} (y \mid \theta)$ are different.
Similar situations also occur in nonlinear regression problems.
\citet{KK08jmva} and \citet{GX08et} showed that shrinkage priors are useful for constructing
Bayesian predictive densities for
linear regression models when the observations are normally distributed with known variance.
However, it has been difficult to construct useful priors for general models other than the normal models
when $x$ and $y$ have different distributions.

In the present paper, we study asymptotic theory for the setting in which $x(i)$, $i=1,\ldots,N$, and $y$
have different distributions.
Although several asymptotic properties of predictive distributions for such a setting are studied by \citet{FKA04sjs},
the result corresponding to \eqref{main0} has not been explored.
The generalization is not straightforward because
two different differential geometric structures,
one for $p(x \mid \theta)$ and the other for $\py(y \mid \theta)$,
such as the Fisher--Rao metrics exist in the present setting.

We introduce a new metric $\gxy{ij}{}$, which we call the predictive metric,
depending on both $p(x \mid \theta)$ and $\py(y \mid \theta)$.
The predictive metric $\gxy{ij}{}$ and the volume element $|\gxy{}{}|^{1/2} \dd \theta^1 \cdots \dd \theta^d$
of it correspond to the Fisher--Rao metric
and the Jeffreys prior in the conventional setting.

In Section 2,
we obtain an expansion of the difference of the risk functions of Bayesian predictive densities.
Each term in the expansion is represented by using geometrical quantities and is invariant with respect to parameter transformations.
In Section 3, we introduce the predictive metric $\gxy{ij}{}$ and evaluate
the asymptotic risk difference between a Bayesian predictive density
based on a prior $\pi$ and that based on the volume element prior $|\gxy{}{}|^{1/2} \dd \theta^1 \cdots \dd \theta^d$ of the predictive metric
$\gxy{ij}{}$.
The asymptotic risk difference is represented by using the Laplacian associated with the predictive metric $\gxy{ij}{}$.
In Section 4, we consider three examples and construct superior priors by using
the formula obtained in Section 3.

\section{An expansion of the risk of predictive densities}
First, we prepare several information geometrical notations to be used.
In the following,
the quantities associated with the model $\{p(x \mid \theta) \, | \, \theta \in \Theta \}$
are denoted without tilde,
and
those associated with the model $\{\py(y \mid \theta) \, | \, \theta \in \Theta)\}$ are denoted with tilde.
We put $l := \log p(x \mid \theta)$ and $\ly := \log \py(y \mid \theta)$.
The Fisher--Rao metrics on the model manifolds $\{p(x \mid \theta) \, | \, \theta \in \Theta \}$
and $\{\py(y \mid \theta) \, | \, \theta \in \Theta \}$ are given by
\begin{gather*}
\gx{ij}{}(\theta) := \E \left(\partial_i \lx \partial_j \lx \, \big| \, \theta \right) ~~~ \mbox{and} ~~~
\gy{ij}{}(\theta) := \E \bigl(\partial_i \ly \partial_j \ly \, \big\vert \, \theta \bigr),
\end{gather*}
respectively.
The $(i,j)$-elements of the inverses of the $d \times d$ matrices $(\gx{ij}{})$ and $(\gy{ij}{})$ are denoted by
$\gx{}{ij}$ and $\gy{}{ij}$, respectively.
We define
\begin{gather*}
\Tx{ijk}{}(\theta) :=
\E \Bigl(\partial_i \lx \partial_j \lx \partial_k \lx \, \Big| \, \theta \Bigr), ~~~
\Ty{ijk}{}(\theta) :=
\E \Bigl(\partial_i \ly \partial_j \ly \partial_k \ly \, \Big| \, \theta \Bigr) , \\
 \eGammax{ijk}{}(\theta) := \E \Bigl( \partial_i \partial_j \lx \partial_k \lx \, \Big| \, \theta \Bigr), ~~~
 \eGammay{ijk}{}(\theta) := \E \Bigl( \partial_i \partial_j \ly \partial_k \ly \, \Big| \, \theta \Bigr), \\
 \mGammax{ijk}{}(\theta) := \E \Bigl( \partial_i \partial_j \lx \partial_k \lx \, \Big| \, \theta \Bigr) + \Tx{ijk}{}(\theta), ~~~
 \mGammay{ijk}{}(\theta) := \E \Bigl( \partial_i \partial_j \ly \partial_k \ly \, \Big| \, \theta \Bigr) + \Ty{ijk}{}(\theta), \\
 \oGammax{ijk}{}(\theta) := \frac{1}{2} \Bigl\{ \eGammax{ijk}{}(\theta) + \mGammax{ijk}{}(\theta) \Bigr\}
 = \frac{1}{2} \left\{\partial_i \gx{jk}{}(\theta) + \partial_j \gx{ki}{}(\theta) - \partial_k \gx{ij}{}(\theta)\right\}, \\
 \oGammay{ijk}{}(\theta) := \frac{1}{2} \Bigl\{\eGammay{ijk}{}(\theta) + \mGammay{ijk}{}(\theta) \Bigr\}
 = \frac{1}{2} \left\{\partial_i \gy{jk}{}(\theta) + \partial_j \gy{ki}{}(\theta) - \partial_k \gy{ij}{}(\theta)\right\},
\end{gather*}
and
\begin{align*}
\Qy{ijkl}{}(\theta) :=
\E \Bigl(\partial_i \ly \partial_j \ly \partial_k \ly \partial_l \ly \, \Big| \, \theta \Bigr).
\end{align*}
Here,
$\eGammax{ijk}{}$ are the e-connection coefficients,
$\mGammax{ijk}{}$ are the m-connection coefficients, and $\oGammax{ijk}{}$ are the Riemannian connection coefficients.
The relations
\begin{align}
\label{duality}
\partial_i \gx{jk}{} = \eGammax{ijk}{} + \mGammax{ikj}{}, ~~~\mbox{and}~~~
\partial_i \gy{jk}{} = \eGammay{ijk}{} + \mGammay{ikj}{}
\end{align}
represent the duality between $\eGammax{ijk}{}$ and $\mGammax{ijk}{}$ with respect to the metric $\gx{ij}{}$,
and the duality between $\eGammay{ijk}{}$ and $\mGammay{ijk}{}$ with respect to the metric $\gy{ij}{}$, respectively.

Covariant derivatives $\enablax{i}{} u^j$, $\onablax{i}{} u^j$, and $\mnablax{i}{} u^j$
of a vector field $u^j$ with respect to the connection coefficients $\eGammax{ij}{k}$, $\oGammax{ij}{k}$, and $\mGammax{ij}{k}$
are defined by $\enablax{i}{} u^j := \partial_i u^j + \sum_k \eGammax{ik}{j} u^k$,
$\onablax{i}{} u^j := \partial_i u^j + \sum_k \oGammax{ik}{j} u^k$,
and $\enablax{i}{} u^j := \partial_i u^j + \sum_k \eGammax{ik}{j} u^k$, respectively,
where $\eGammax{ik}{j} = \sum_l \eGammax{ikl}{} \gx{}{jl}$, $\oGammax{ik}{j} = \sum_l \oGammax{ikl}{} \gx{}{jl}$,
and $\mGammax{ik}{j}= \sum_l \mGammax{ikl}{}\gx{}{jl}$.
In the same way, the covariant derivatives $\enablay{i}{} u^j$, $\onablay{i}{} u^j$, and $\mnablay{i}{} u^j$,
with respect to the connection coefficients
$\eGammay{ik}{j} = \sum_l \eGammay{ikl}{} \gy{}{jl}$, $\oGammay{ik}{j} = \sum_l \oGammay{ikl}{} \gy{}{jl}$,
and $\mGammay{ik}{j} = \sum_l \mGammay{ikl}{} \gy{}{jl}$, are defined.

Theorem \ref{riskdiff} below is used in the following sections.

\vspace{0.3cm}
\noindent
\begin{thm} \label{riskdiff}
The difference between the risk functions of Bayesian predictive densities $\py_{\pi}(y \mid x^N)$
and $\py_{\pi'}(y \mid x^N)$
based on priors $\pi(\theta) \dd \theta$ and $\pi'(\theta) \dd \theta$, respectively,
is given by
\begin{align}
N^2 \biggl[ \E \big\{ & D (\py(y \mid \theta), \py_{\pi}(y \mid x^N) \, \big| \, \theta \big\}
- \E \big\{ D (\py(y \mid \theta), \py_{\pi'}(y \mid x^N) \, \big| \, \theta \big\} \biggr] \nonumber \\
=&
\left(
\frac{1}{2} \sum_{i, j}  \gy{ij}{} \upi^i \upi^j
+ \sum_{i, j, k}  \gy{ij}{} \gx{}{jk} \enablay{k}{} \upi^i \right) \notag \\
& -
\left(
\frac{1}{2} \sum_{i, j} \gy{ij}{} \upiprime^i \upiprime^j
+ \sum_{i, j, k}  \gy{ij}{} \gx{}{jk} \enablay{k}{} \upiprime^i \right) + \mathrm{o}(1),
\label{6-2}
\end{align}
where
\begin{equation*}
u^i_\pi(\theta) := \sum_k \gx{}{ik}(\theta) \Bigl\{ \partial_k \log \pi(\theta) - \sum_j \eGammax{kj}{j}(\theta) \Bigr\}
+ \sum_{k, l}  \gx{}{kl}(\theta) \Bigl\{ \mGammay{kl}{i}(\theta) - \mGammax{kl}{i}(\theta) \Bigr\}.
\end{equation*}
\end{thm}

\vspace{0.3cm}

The proof of Theorem \ref{riskdiff} is given in the Appendix.
\section{Prior construction based on the predictive metric}
In this section, we introduce a new metric defined by
\begin{equation}
 \gxy{ij}{} := \sum_{k=1}^d \sum_{l=1}^d \gx{ik}{} \gy{}{kl} \gx{jl}{}, \label{10-1}
\end{equation}
which we call the predictive metric.
Since $\gxy{ij}{}$ is positive definite, it can be adopted as a Riemannian metric on $\Theta$.
It will be shown that the predictive metric $\gxy{ij}{}$, the corresponding volume element
\begin{align}
\piP(\theta) \dd \theta := | \gxy{ij}{} |^{\frac{1}{2}} \dd \theta
= \bigl| \sum_{k=1}^d \sum_{l=1}^d  \gx{ik}{} \gy{}{kl} \gx{lj}{} \bigr|^{\frac{1}{2}} \dd \theta
= | \gx{ij}{} | | \gy{ij}{} |^{-\frac{1}{2}} \dd \theta,
\label{piP}
\end{align}
and the Laplacian $\Deltaxy$ based on $\gxy{ij}{}$
play essential roles corresponding to those played by
the Fisher--Rao metric $g_{ij}$, the Jeffreys prior $|g_{ij}|^{1/2} \dd \theta$,
and the Laplacian $\Delta$ based on $g_{ij}$
in the conventional setting where $\gx{ij}{} = \gy{ij}{}$.
Here, $|\gxy{ij}{}|$, $|g_{ij}|$, and $|\tilde{g}_{ij}|$ denote determinants of $d \times d$ matrices
$(\gxy{ij}{})$, $(g_{ij})$, and $(\tilde{g}_{ij})$, respectively.
The $(i,j)$-element of the inverse of the $d \times d$ matrix $(\gxy{ij}{})$ is given by
$\gxy{}{ij} := \sum_{k, l}  \gy{kl}{} \gx{}{ik} \gx{}{jl}$.

Here, we give an intuitive meaning of the predictive metric $\gxy{ij}{}$ by a nonrigorous argument.
In the standard estimation theory,
the Fisher-Rao metric $g_{ij}$, which is the Fisher information matrix,
corresponds to the inverse
of the asymptotic variance of the maximum likelihood estimator.
In the setting we consider,
the asymptotic variance of the maximum likelihood estimator based on $x^N$
is $(N \gx{}{})^{-1}$, where $g$ is the $d \times d$ matrix $(g_{ij})$,
and the asymptotic variance of the maximum likelihood estimator
based on both of $x^N$ and $y$ is $(N \gx{}{} + \gy{}{})^{-1}$,
where $\tilde{g}$ is the $d \times d$ matrix $(\tilde{g}_{ij})$.
The inverse of the reduction of the asymptotic variance by observing $y$ in addition to $x(i)$ $(i=1,\ldots,N)$
are given by
$\{(N g)^{-1}-(N \gx{}{} + \gy{}{})^{-1}\}^{-1} = N^2 \gxy{}{} + \mathrm{O}(N)$,
as we see in Example 1 in Section 4,
corresponding to the predictive metric $\gxy{}{}$.

The Riemannian connection coefficients with respect to the predictive metric $\gxy{ij}{}$ are given by
\[
\oGammayx{ijk}{} = \frac{1}{2} \left( \partial_i \gxy{jk}{} + \partial_j \gxy{ki}{} - \partial_{k} \gxy{ij}{} \right),
\]
and we put $\oGammaxy{ik}{j} = \sum_l \oGammaxy{ikl}{} \gxy{}{jl}$.
Then,
\begin{align}
\partial_k \log |\gxy{ij}{}|^{\frac{1}{2}} =& \frac{1}{2} \partial_k \log |\gxy{ij}{}|
= \sum_{i, j}  \frac{1}{2} (\partial_k \gxy{ij}{}) \gxy{}{ij}
= \sum_i \oGammayx{ki}{i}. \label{4-0}
\end{align}
In the same way, we have
\begin{align}
\partial_k \log |\gx{ij}{}|^{\frac{1}{2} }
= \sum_i \oGammax{ki}{i}, ~~ \mbox{and} ~~~~
\partial_k \log |\gy{ij}{}|^{\frac{1}{2} }
= \sum_i \oGammay{ki}{i}.
\label{9-2}
\end{align}
Thus,
\begin{align}
\sum_i \oGammayx{ki}{i} =& \partial_k \log | \gxy{ij}{} |^{\frac{1}{2} }
= \partial_k \log |\gx{ij}{}| - \frac{1}{2} \partial_k \log |\gy{ij}{}|
= 2 \sum_{i} \oGammax{ki}{i} - \sum_{i} \oGammay{ki}{i}.
\label{9-1}
\end{align}

The Laplacian $\Deltaxy$ with respect to the predictive metric $\gxy{ij}{}$ is defined by
\begin{align*}
\Deltaxy f = \sum_{i, j}  \onablaxy{i}{} (\gxy{}{ij} \partial_j f)
=& \sum_{i, j}  \partial_i (\gxy{}{ij} \partial_j f) + \sum_{i, j, k}  \oGammayx{ik}{i} \gxy{}{kj} \partial_j f \\
=& \sum_{i, j}  \gxy{}{ij} (\partial_i \partial_j f - \sum_k \oGammayx{ij}{k} \partial_k f),
\end{align*}
where $\onablaxy{i}{} u^j = \partial_i u^j + \sum_k \oGammaxy{ik}{j} u^k$,
and
$f$ is a real smooth function on $\Theta$.

By using these quantities, we obtain the following theorem corresponding to \eqref{main0}
in the conventional setting.

\begin{thm}
\label{maintheorem}
The difference between the risk functions of Bayesian predictive densities $\py_{\pi}(y \mid x^N)$ based on
a $\pi(\theta) \dd \theta$ and $\py_{\mathrm P}(y \mid x^N)$ based on
$\pi_\mathrm{P}(\theta) \dd \theta$
is given by\vadjust{\goodbreak}
\begin{align}
N^2 \biggl[ \E \big\{ & D (\py(y \mid \theta), \py_{\pi}(y \mid x^N) \, \big| \, \theta \big\}
- \E \big\{ D (\py(y \mid \theta), \py_{\mathrm P}(y \mid x^N)  \, \big| \, \theta \big\} \biggr] \notag\\
=& \frac{\displaystyle \Deltaxy \left(\frac{\pi}{\piP}\right)}{\displaystyle \left( \frac{\pi}{\piP} \right) }
- \frac{1}{2} \sum_{i, j}
\gxy{}{ij} \frac{\displaystyle \partial_i \left(\frac{\pi}{\piP}\right) \partial_j \left(\frac{\pi}{\piP}\right)}
{\displaystyle \left(\frac{\pi}{\piP}\right)^2}
+ \oo(1)
= 2 \; \frac{\displaystyle \Deltaxy \left(\frac{\pi}{\piP}\right)^{\frac{1}{2}}}
{\displaystyle \left( \frac{\pi}{\piP} \right)^{\frac{1}{2}}}
+ \oo(1).\label{main}
\end{align}
\end{thm}

The proof of Theorem \ref{maintheorem} is given in the Appendix.

If there exists a positive constant $c$ such that $\pi' = c \pi$,
we identify the prior $\pi'$ with $\pi$
because the posterior densities based on them are identical.
In fact, the risk difference \eqref{main} between $\pi$ and $\pi_\mathrm{P}$ coincides with
that between $\pi'$ and $\pi_\mathrm{P}$.

\vspace{0.2cm}
\begin{cor}~
\label{usefulcor}
If a positive function $f(\theta)$ is superharmonic
with respect to the predictive metric $\gxy{}{}$, i.e.\ $\Deltaxy f(\theta) \leq 0$ for every $\theta \in \Theta$,
and the strict inequality holds at a point in $\Theta$,
then the Bayesian predictive density based on the prior density $\{f(\theta)\}^2 \piP (\theta)$ asymptotically dominates
the Bayesian predictive density $\py_{\mathrm P}(y \mid x^N)$ based on the prior density  $\piP(\theta)$.
If there exists a non-constant positive superharmonic function $f(\theta)$
with respect to the predictive metric $\gxy{}{}$,
then the Bayesian predictive density based on the prior density $\{f(\theta)\}^{2c} \piP (\theta)$
$(0 < c < 1)$
asymptotically dominates
$\py_{\mathrm P}(y \mid x^N)$.
\end{cor}\vspace{-0.2cm}
\begin{proof}
The first statement
is a straightforward conclusion from Theorem \ref{maintheorem}.
We show the second statement.
The function
$\{ f ( \theta)\}^c \, (0 < c < 1)$ is superharmonic
because
$\Delta f^c = cf^{ -(1-c)} \{ \Delta f - (1 - c) f^{-1} g^{ij} \partial_i f \partial_j f \} \le 0$
if $f(\theta)$ is a positive superharmonic function.
The strict inequality holds at $\theta$ satisfying $\partial_i f(\theta) \neq 0$ for any $i$.
Such $\theta$ exists since $f(\theta)$ is a non-constant function.
Thus, the second statement follows from the first statement.
\end{proof}
\vspace{-0.2cm}

By setting $c = 1/2$, it follows from Corollary \ref{usefulcor} that
the Bayesian predictive density based on the prior $f(\theta) \pi_\mathrm{P}(\theta)$
asymptotically dominates the Bayesian predictive density based on $\pi_\mathrm{P}$
if $f(\theta)$ is a non-constant positive superharmonic function.

Note that Corollary \ref{usefulcor} also holds if we replace the predictive metric $\mathring{g}$ with another
metric $\mathring{g}'$ satisfying $\mathring{g}' = c \mathring{g}$ with a positive constant $c$.
This is because the volume element with respect to $\mathring{g}'$ is proportional to
that with respect to $\mathring{g}$
and the relation $\mathring{\Delta}' f = (1/c) \mathring{\Delta} f$ holds,
where $\mathring{\Delta}'$ is the Laplacian with respect to
$\mathring{g}'$.

\section{Examples}
In this section, we see three examples.
We verify that the results in the previous sections are consistent with several known results in Examples 1 and 2
and obtain some new results in Examples 2 and 3.

\vspace{0.1cm}
\noindent
{Example 1. Normal models}

Suppose that $x$ is distributed according to the $d$-dimensional normal distribution $\mathrm{N}_d (\mu, \Sigma)$
with mean vector $\mu$ and covariance\vadjust{\goodbreak} matrix $\Sigma=(\Sigma^{ij})$ and that $y$ is distributed\ according to
the $d$-dimensional normal distribution $\mathrm{N}_d (\mu, \tilde{\Sigma})$ with the same mean vector $\mu$ and
possibly different covariance matrix $\tilde{\Sigma}=(\tilde{\Sigma}^{ij})$.
Here, $\mu$ is the unknown parameter and $\Sigma$ and $\tilde{\Sigma}$ are known.

The Fisher information matrix for $p(x \mid \mu)$
is $(\gx{ij}{}) = (\Sigma_{ij})$
and that for $p(x \mid \mu)$ is $(\gy{ij}{}) = (\tilde{\Sigma}_{ij})$,
where $(\Sigma_{ij})$ and $(\tilde{\Sigma}_{ij})$ are inverse matrices
of $(\Sigma_{ij})$ and $(\tilde{\Sigma}_{ij})$, respectively.

Since the coefficients of the predictive metric
$\gxy{ij}{} = \sum_{k, l}  \gx{ik}{} \gy{}{kl} \gx{jl}{}$
do not depend on $\mu$,
the volume element with respect to the predictive metric is
\[
 \piP(\mu) \dd \mu = \left| \gxy{ij}{} \right| \dd \mu \propto \dd \mu,
\]
which is the uniform distribution $\pi_\mathrm{U}(\mu) \propto 1$.

\citet{KK08jmva} and \cite{GX08et}
considered shrinkage pri-\break ors for this model.
The Bayesian predictive density $\py_\pi(y \mid x^N)$ dominates
$\py_\mathrm{U}(y \mid x^N)$ based on the uniform measure $\pi_\mathrm{U}(\mu)$
if $\pi(\mu)$ is a superharmonic function on the Euclidean space $\mathbb{R}^d$
endowed with the metric $\big( ( N \, \gx{}{})^{-1} - ( N \gx{}{} + \gy{}{})^{-1} \big)^{-1}$,
see Theorem 3.2 in \citet{KK08jmva}.
This result holds for every positive integer $N$.

Since
\begin{align*}
\big( ( N \gx{}{})^{-1} - ( N \gx{}{} + \gy{}{})^{-1} \big)^{-1}
=& (N\gx{}{})^{\frac{1}{2}}
\left[ I - \left\{  I + (N\gx{}{})^{ - \frac{1}{2}} \gy{}{} ( N\gx{}{} ) ^{-\frac{1}{2}} \right\}^{-1} \right]^{-1}
(N\gx{}{})^{\frac{1}{2}} \\[2pt]
=& (N\gx{}{})^{\frac{1}{2}}
\left[ I - I + (N\gx{}{})^{ -\frac{1}{2}} \gy{}{} ( N\gx{}{})^{ -\frac{1}{2}} + \mathrm{O}(N^{-2}) \right]^{-1}
(N\gx{}{})^{\frac{1}{2}}   \\[2pt]
=& N^2 \, \gx{}{} \, \gy{}{-1} \, \gx{}{} + \mathrm{O}(N)
\end{align*}
corresponds to the predictive metric $\gxy{}{}$,
Theorem \ref{maintheorem} is consistent with theoretical and numerical results in \citet{KK08jmva} and \cite{GX08et}.

\vspace{0.3cm}
\noindent
{Example 2. Location-scale models}

Suppose that $\phi(x)$ and $\tphi(y)$ are probability densities on $\mathbb{R}$ that are symmetric about the origin.
Let
\begin{align*}
p(x \mid \mu, \sigma) \dd x :=& \frac{1}{\sigma } \phi \left(\frac{x - \mu}{\sigma} \right) \dd x ~~~ \mbox{and} ~~~~
\py(y \mid \mu, \sigma) \dd y := \frac{1}{\sigma } \tphi \left(\frac{y - \mu}{\sigma} \right) \dd y,
\end{align*}
where $\mu \in \mathbb{R}$ and $\sigma > 0$ are unknown parameters.

Suppose that we have a set of $N$ independent observations $x(1),\ldots,x(N)$ distributed according to $p(x \mid \mu, \sigma)$.
The variable $y$ to be predicted is independently distributed according to $\py(y \mid \mu, \sigma)$.
The objective is to construct a prior $\pi$ for a Bayesian predictive density $\py_\pi(y \mid x)$.

The Fisher--Rao metrics on the model manifolds $\{p(x \mid \mu, \sigma)\}$ and $\{\py(y \mid \mu, \sigma)\}$ are
\begin{align*}
\gx{\mu \mu}{} =& \frac{a}{\sigma^2}, ~~ \gx{\sigma \sigma}{} = \frac{b}{\sigma^2}, ~~ \gx{\mu \sigma}{} = 0,\\
\gy{\mu \mu}{} =& \frac{\tilde{a}}{\sigma^2}, ~~ \gy{\sigma \sigma}{} = \frac{\tilde{b}}{\sigma^2},  ~ \text{and} ~ \gy{\mu \sigma}{} = 0,
\end{align*}
respectively,
where $a$ and $b$ are positive constants depending on $\phi(x)$,
and $\tilde{a}$ and $\tilde{b}$ are positive constants depending on $\tphi(y)$.

The predictive metric is given by
\begin{align*}
\gxy{\mu \mu}{} =& \frac{a^2/\tilde{a}}{\sigma^2}, ~~ \gxy{\sigma \sigma}{} = \frac{b^2/\tilde{b}}{\sigma^2},  ~~ \text{and} ~~~ \gxy{\mu \sigma}{} = 0.
\end{align*}
Define
\begin{align}
\label{rescale}
u := \sqrt{ \frac{\tilde{b}}{\tilde{a}}} \frac{a}{b} \mu, ~~~ v := \sigma
\end{align}
by rescaling the location parameter $\mu$.
We call this coordinate system $(u,v)$ the upper-half plane coordinates.
Then, the predictive metric is represented by
\begin{align*}
\gxy{uu}{} &= \frac{b^2/\tilde{b}}{v^2}, ~~~ \gxy{vv}{} = \frac{b^2 / \tilde{b}}{v^2},
~\mbox{and} ~~ \gxy{uv}{} = 0,
\end{align*}
coinciding with the metric on the Hyperbolic plane $H^2 (-\tilde{b}/b^2)$,
which is a 2-dimensional complete manifold with constant sectional curvature $- \tilde{b}/b^2$.
Thus, the model manifold endowed with the predictive metric $\gxy{}{}$ is isometric to $H^2 (-\tilde{b}/b^2)$.

The volume element with respect to the predictive metric $\gxy{}{}$ is given by
\[
\piP(\mu,\sigma) \dd \mu \dd \sigma = |\gxy{}{}|^{1/2} \dd \mu \dd \sigma \propto \frac{1}{\sigma^2} \dd \mu \dd \sigma
\]
and coincides with the Jeffreys priors
$|\gx{}{}|^{1/2} \dd \mu \dd \sigma \propto 1/\sigma^2 \dd \mu \dd \sigma$
for $p(x \mid \mu, \sigma)$ and $|\gy{}{}|^{1/2} \dd \mu \dd \sigma \propto 1/\sigma^2 \dd \mu \dd \sigma$
for $\py(y \mid \mu, \sigma)$.

The Laplacian on the model manifold endowed with the predictive metric $\gxy{}{}$ is given by
\begin{align}
\Deltaxy = \sigma^2 \left(\frac{\tilde{a}}{a^2}
\frac{\partial^2}{\partial \mu^2} + \frac{\tilde{b}}{b^2} \frac{\partial ^2}{\partial \sigma^2} \right).
\label{laplacian-org}
\end{align}

By Corollary \ref{usefulcor},
the Bayesian predictive density $\py_\mathrm{R}(y \mid x)$ based on the prior
\[
 \pi_\mathrm{R}(\mu,\sigma) \dd \mu \dd \sigma \propto \frac{1}{\sigma} \dd \mu \dd \sigma
\]
asymptotically dominates $\py_\mathrm{P}(y \mid x)$ based on $\piP$ because
\[
\Deltaxy \frac{\pi_\mathrm{R}(\mu,\sigma)}{\piP(\mu,\sigma)} =
\Deltaxy \frac{1/\sigma}{1/\sigma^2} = \Deltaxy \sigma
= \sigma ^2 \left(\frac{\tilde{a}}{a^2} \frac{\partial ^2}{\partial \mu^2} + \frac{\tilde{b}}{b^2} \frac{\partial ^2}{\partial \sigma ^2} \right) \sigma = 0.
\]
By Theorem \ref{maintheorem}, the asymptotic risk difference is
\begin{align}
N^2 \biggl[ \E \big\{ & D (\py(y \mid \theta), \py_\mathrm{R}(y \mid x^N) \, \big| \, \theta \big\}
- \E \big\{ D (\py(y \mid \theta), \py_\mathrm{P}(y \mid x^N)  \, \big| \, \theta \big\} \biggr] \notag\\
=& 2 \; \frac{\displaystyle \Deltaxy \left(\frac{\pi_\mathrm{R}}{\piP}\right)^{\frac{1}{2}}}
{\displaystyle \left( \frac{\pi_\mathrm{R}}{\piP} \right)^{\frac{1}{2}}}
+ \oo(1)
= 2 \; \frac{\Deltaxy \sigma^{\frac{1}{2}}}
{\sigma^{\frac{1}{2}}}
+ \oo(1)
= - \frac{\tilde{b}}{2 b^2} + \oo(1).
\label{riskr}
\end{align}

In fact, it can be shown that the Bayesian predictive density $\py_\mathrm{R}(y \mid x)$ exactly dominates
$\py_\mathrm{P}(y \mid x)$ for finite $N$
because $\piP$ is the left invariant prior and $\pi_\mathrm{R}$ is the right invariant prior
with respect to the location-scale group.
The Bayesian procedures based on the right invariant prior
dominate those based on the left invariant prior
in many problems associated with group models as shown in \citet{Z68aism}.
The prior $\pi_\mathrm{R}$ is also derived as a reference prior, see \cite{BB92bs}.

Furthermore,  as we see below, the Bayesian predictive density $\py_{c,\kappa}(y \mid x)$ based on the prior $\pi_{c,\kappa}$ defined by
\begin{align}
 \frac{\pi_{c, \kappa}}{\piP}(\mu, \sigma)
:=& \frac{2 \kappa \sigma}
{\displaystyle \frac{a^2 \tilde{b}}{b^2 \tilde{a}} \mu^2 + c(\sigma + \kappa)^2 + (1-c)(\sigma^2 + \kappa^2)}
~~~~ (0 \leq c \leq 1, ~ 0 < \kappa < \infty)
\label{hs}
\end{align}
asymptotically dominates $\py_\mathrm{R}(y \mid x)$
and thus also dominates $\py_\mathrm{P}(y \mid x)$.

To clarify the meaning of the prior $\pi_{c,\kappa}$,
we introduce another coordinate system on the model manifold.
Let $(b/\sqrt{\tilde{b}}) \rho$ be the Riemannian distance based on the predictive metric $\gxy{}{}$ between
a point $\mathrm{P}$ and an arbitrary fixed point $\mathrm{O}$ on $H^2(-\tilde{b}/b^2)$.
The direction of $\mathrm{P}$ from $\mathrm{O}$ is represented by a point $\tau$ on the unit circle in the tangent space at $\mathrm{O}$.
Then, the point $\mathrm{P}$ is represented by $\rho$ and $\tau$,
see e.g.\ \cite{H84ap} p.~152.
This coordinate system $(\rho,\tau)$ is called the geodesic polar coordinates.
Then, the predictive metric is given by
\[
\gxy{\rho \rho}{} = \frac{b^2}{\tilde{b}}, ~~~
\gxy{\tau \tau}{} = \frac{b^2}{\tilde{b}}(\sinh \rho)^2, ~\mbox{and}~~
\gxy{\rho \tau}{} = 0.
\]
The Laplacian is represented by
\begin{align}
\label{polarlaplacian}
\Deltaxy = \frac{\tilde{b}}{b^2} \left\{ \frac{\partial^2}{\partial \rho^2}
+ \frac{ \cosh \rho}{\sinh \rho} \frac{ \partial}{\partial \rho} + (\sinh \rho)^{-2} \Deltaxy_\mathrm{S} \right\},
\end{align}
where $\Deltaxy_\mathrm{S}$ is the Laplacian on the unit circle in the tangent space at $\mathrm{O}$,
see e.g. \cite{H84ap} p.~158.

When the upper-half plane coordinate system is adopted,
the Riemannian distance $(b/\sqrt{\tilde{b}}) \rho$ between $(u,v)$ and $(\bar{u},\bar{v})$ is represented by
\[
\cosh \rho = \frac{ | u - \bar{u} |^2 + v^2 + \bar{v}^2 }{2 v \bar{v}},
\]
see e.g.\ \cite{D89cambridge} p. 176.
Thus,
in the original coordinate system $(\mu,\sigma)$, the\break Riemannian distance $(b/\sqrt{\tilde{b}}) \rho$ between
and $(\mu, \sigma)$ and $(0, \kappa)$ is
\begin{align}
\label{coshrho}
\cosh \rho = \frac{\frac{a^2\tilde{b}}{b^2\tilde{a}} \mu^2 + \sigma^2 + \kappa^2}{2\sigma \kappa}.
\end{align}
Thus, the ratio of prior densities is given by
\begin{align}
\frac{\pi_{c,\kappa}(\mu,\sigma)}{\pi_\mathrm{P}(\mu,\sigma)}
\label{ckapparatio-org}
=& \frac{1}{\displaystyle \frac{\frac{a^2\tilde{b}}{b^2\tilde{a}} \mu^2 + \sigma^2 + \kappa^2}{2\sigma \kappa}+c} \\
=& \frac{1}{\cosh \rho + c}.
\label{ckapparatio}
\end{align}
Note that $\pi_{c,\kappa}(\mu,\sigma)/\pi_\mathrm{P}(\mu,\sigma)$
depends on $(\mu,\sigma)$ only through $\rho(\mu,\sigma)$ defined by \eqref{coshrho}.
Thus, from \eqref{polarlaplacian}, \eqref{ckapparatio}, and Theorem \ref{maintheorem}, we have
\begin{align}
N^2 \biggl[ \E \big\{ & D (\py(y \mid \theta), \py_{\pi_{c,\kappa}}(y | x ) \, \big| \, \theta \big\}
- \E \big\{ D (\py(y \mid \theta), \py_\mathrm{P}(y \mid x )  \, \big| \, \theta \big\} \biggr] \notag\\
=& 2 \; \frac{\displaystyle \Deltaxy \left(\frac{\pi_{c,\kappa}}{\piP}\right)^{\frac{1}{2}}}
{\displaystyle \left( \frac{\pi_{c,\kappa}}{\piP} \right)^{\frac{1}{2}}}
+ \oo(1)
= - \frac{\tilde{b}}{b^2}
\left\{
\frac{1}{2} + c \frac{\pi_{c,\kappa}}{\piP} + \frac{3}{2} (1-c^2) \left(\frac{\pi_{c,\kappa}}{\piP} \right)^2 \right\}
+ \oo(1) \notag \\
=& - \frac{\tilde{b}}{b^2} \left\{ \frac{1}{2} + c \frac{1}{\cosh \rho + c} + \frac{3}{2} ( 1 - c^2) \frac{1}{(\cosh \rho + c)^2} \right\}
+ \oo(1),
\label{riskhs}
\end{align}
and \eqref{riskhs} is smaller than \eqref{riskr}
when $0 \leq c < 1$ and $0 < \kappa < \infty$.
The asymptotic risk difference \eqref{riskhs} can also be derived from \eqref{ckapparatio-org}
and the Laplacian \eqref{laplacian-org} in the original coordinate system.

By Corollary \ref{usefulcor},
the Bayesian predictive density $\py_{c,\kappa}(y \mid x^N)$
$(0 \leq c < \infty, 0 < \kappa < \infty)$
asymptotically dominates $\py_\mathrm{P}(y \mid x^N)$
since the function \eqref{hs} is superharmonic for $0 \leq c < \infty$.
However, $\py_{c,\kappa}(y \mid x^N)$
asymptotically dominates $\py_\mathrm{R}(y \mid x^N)$ only when $0 \leq c \leq 1$.
\begin{figure}[ht]
  \includegraphics{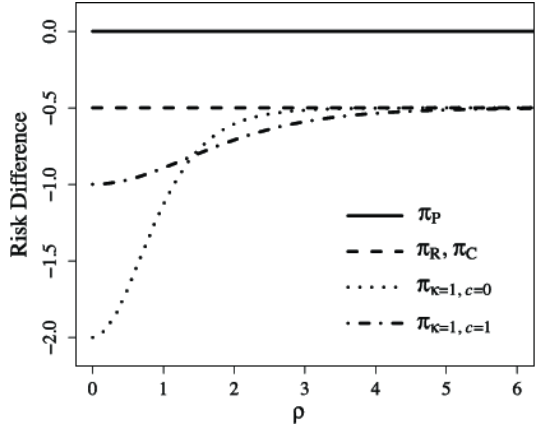}
  \caption{The asymptotic risk difference
$N^2[\E\{ D(\py(y \mid \theta), \py_{c,\kappa}(y \mid x^N) \mid \theta \}
- \E \{ D(\py(y \mid \theta), \py_\mathrm{P}(y \mid x^N) \mid \theta \}] + \oo(1)
= -(\tilde{b}/b^2)
\{
1/2 + c (\pi/\piP) + (3/2) (1-c^2) (\pi/\piP) \}^2$
for Bayesian predictive densities based on $\pi_\mathrm{P}$, $\pi_\mathrm{R}$, $\pi_\mathrm{C}$,
$\pi_{\kappa = 1, c= 0}$, and $\pi_{\kappa = 1, c= 1}$.
We put $\tilde{b}/b^2 = 1$ just for simplicity.}
  \label{fig:location-scale-risk}
\end{figure}

Several properties of the function \eqref{hs} are discussed in \citet{K07aism}.
As $\kappa \to \infty$, the prior $\pi_{c, \kappa}$ converges to the right invariant prior $\pi_\mathrm{R}$,
because
\begin{align*}
\frac{\pi_{c,\kappa}(\mu,\sigma)}{\pi_\mathrm{P}(\mu,\sigma)} =
\frac{1}{\displaystyle \frac{\frac{a^2\tilde{b}}{b^2\tilde{a}} \mu^2 + \sigma^2 + \kappa^2}{2\sigma \kappa}+c}
\propto
\frac{\kappa/2}{\displaystyle \frac{\frac{a^2\tilde{b}}{b^2\tilde{a}} \mu^2 + \sigma^2 + \kappa^2}{2\sigma \kappa}+c}
\rightarrow
\sigma
\end{align*}
when $\kappa \to \infty$.
Here, priors are identified up to a positive multiplicative constant.
As $\kappa \to 0$, the prior $\pi_{c, \kappa}$ converges to
\begin{align*}
\pi_\mathrm{C}(\mu,\sigma) \dd \mu \dd \sigma& :=
\frac{\sigma}{ \frac{a^2 \tilde{b}}{b^2 \tilde{a}} \mu^2 + \sigma^2} \frac{1}{\sigma^2} \dd \mu \dd \sigma
= \frac{\sigma^{-1}}{ \frac{a^2 \tilde{b}}{b^2 \tilde{a}}( \frac{ \mu}{\sigma})^2 +1} \frac{1}{\sigma^2} \dd \mu \dd \sigma,
\end{align*}
because
\begin{align*}
\frac{\pi_{c,\kappa}(\mu,\sigma)}{\pi_\mathrm{P}(\mu,\sigma)} =
\frac{1}{\displaystyle \frac{\frac{a^2\tilde{b}}{b^2\tilde{a}} \mu^2 + \sigma^2 + \kappa^2}{2\sigma \kappa}+c}
\propto
\frac{1/(2\kappa)}{\displaystyle \frac{\frac{a^2\tilde{b}}{b^2\tilde{a}} \mu^2 + \sigma^2 + \kappa^2}{2\sigma \kappa}+c}
\rightarrow
\frac{\sigma}{\frac{a^2 \tilde{b}}{b^2 \tilde{a}} \mu^2 + \sigma^2}
\end{align*}

\noindent when $\kappa \to 0$.
The prior density with respect to the rescaled parameter $(u,v)$ defined by \eqref{rescale} is given by
\begin{align}
 \pi_\mathrm{C}(\mu,\sigma) \dd \mu \dd \sigma
\propto \frac{v^{-1}}{(u/v)^2 + 1} \frac{1}{v^2} \dd u \dd v.
\label{limit0}
\end{align}
Note that the Cauchy prior for $u$, discussed by Jeffreys and many researchers, appears in \eqref{limit0}.
Thus, the class $\pi_{c, \kappa}$
of priors bridges the right invariant prior $\pi_\mathrm{R}$,
coinciding with the reference prior, and the Cauchy prior $\pi_\mathrm{C}$.

Figure \ref{fig:location-scale-risk} illustrates the difference between the risk functions of Bayesian predictive densities based on
$\pi_\mathrm{R}$, $\pi_\mathrm{C}$, $\pi_{\kappa=1, c=0}$, and $\pi_{\kappa=1, c =1}$
and the risk function of $\py_\mathrm{P}(y \mid x^N)$.
The risk functions of the right invariant prior $\pi_\mathrm{R}$ and the Cauchy prior $\pi_\mathrm{C}$
are uniformly smaller than that of $\pi_\mathrm{P}$.
The asymptotic risk of the Cauchy prior $\pi_\mathrm{C}$ coincides with that of $\pi_\mathrm{R}$.
Furthermore, the asymptotic risks of $\pi_{\kappa =1, c=0} $ and $\pi_{\kappa= 1, c=1}$ are smaller than that of
$\pi_\mathrm{R}$ for every $(\mu, \sigma)$.
Therefore, the use of $\pi_{\kappa=1, c}$ $(0 \leq c \leq 1)$ is recommended.
The risk of $\pi_{\kappa=1, c=0}$ is smaller than that of $ \pi_{\kappa=1, c=1}$ when $\rho$ is small,
and vice versa.
Thus, there does not exist a unique best value of $c$.
The choice of the value of $0 < \kappa < \infty$ is arbitrary because it corresponds to the center of shrinkage.
Finite-sample decision theoretic properties such as admissibility of Bayesian predictive densities
$\py_{\kappa,c}(y \mid x^N)$
based on proposed priors $\pi_{\kappa, c}$ $(0 < \kappa < \infty,~ 0 \leq c \leq 1)$ require further research.

\vspace{0.3cm}
\noindent
{Example 3. Poisson models}

Suppose that $x_i$ $(i=1,\ldots,d)$ are independently distributed
according to the Poisson distribution $\mbox{Po}(\lambda_i)$ with mean $\lambda_i$
and that $y_i$ $(i=1,\ldots,d)$ are independently distributed
according to the Poisson distribution $\mbox{Po}(s_i \lambda_i)$ with mean $s_i \lambda_i$.
Here, $s_i$ are known positive constants.
The unknown parameter is $\theta = (\theta^1, \ldots, \theta^d) := (\lambda_1, \ldots, \lambda_d)$.
The objective is to construct a predictive density for $y$ by using $x$.
This problem in the conventional setting, in which $s_1 = s_2 = \cdots = s_d$,
is studied in \cite{K04as}.

If $s_i \ll 1$ for each $i$, then this prediction problem is in the asymptotic setting.
The Fisher--Rao metrics corresponding to $x$ and $y$ are given by
\[
\gx{ij}{} =
\left\{
\begin{array}{cc}
\displaystyle \frac{1}{\lambda_i} & (i=j) \\[0.3cm]
0 & (i \neq j)
\end{array}
\right. ~~~ \mbox{and} ~~~~~
\gy{ij}{} =
\left\{
\begin{array}{cc}
\displaystyle \frac{s_i}{\lambda_i} & (i=j) \\[0.3cm]
0 & (i \neq j)
\end{array}
\right.,
\]
respectively.
The predictive metric is
\begin{align}
\label{poissonpm}
\gxy{ij}{} =
\left\{
\begin{array}{cc}
\displaystyle \frac{1}{s_i \lambda_i} & (i=j) \\[0.3cm]
0 & (i \neq j)
\end{array}
\right.,
\end{align}
and the corresponding volume element is
\[
\piP(\lambda) \dd \lambda :=
|\gxy{}{}|^{1/2} \dd \lambda
= \left\{\prod_{i=1}^d \frac{1}{(s_j \lambda_j)^{1/2}}\right\} \dd \lambda
\propto \frac{1}{(\lambda_1 \cdots \lambda_d)^{1/2}} \dd \lambda
\]
coinciding with the Jeffreys priors for $p(x \mid \lambda)$ and $\py(y \mid \lambda)$.

The Laplacian $\Deltaxy{}{}$ based on the predictive metric $\gxy{}{}$ is given by
\[
 \Deltaxy f = \left(\prod_{k=1}^d \lambda_k^{1/2}\right) \sum_{i=1}^d \frac{\partial}{\partial \lambda_i}
\left(\frac{s_i \lambda_i}{\prod_{j=1}^d \lambda_j^{1/2}} \frac{\partial f}{\partial \lambda_i} \right)
= \sum_i s_i \left( \lambda_i \frac{\partial^2 f}{\partial \lambda_i^2} + \frac{1}{2} \frac{\partial f}{\partial \lambda_i} \right),
\]
where $f$ is a smooth real function of $\lambda$.

Define
\[
\pi_\mathrm{S}(\lambda) \dd \lambda := \frac{(\lambda_1/s_1 + \cdots + \lambda_d/s_d)^{-(d/2-1)}}{\prod_j {\lambda_j}^{1/2}} \dd \lambda
\propto (\lambda_1/s_1 + \cdots + \lambda_d/s_d)^{-(d/2-1)} |\gxy{}{}|^{1/2} \dd \lambda.
\]
Then, from
\begin{align}
\label{1stderiv}
\frac{ \partial}{\partial \lambda_i} \frac{\pi_\mathrm{S}(\lambda)}{\pi_\mathrm{P}(\lambda)}
= \left( - \frac{d}{2} +1 \right) \left( \frac{ \lambda_1}{s_1}
+ \dotsb + \frac{\lambda_d}{s_d} \right) ^{ -\frac{d}{2}} \frac{1}{s_i}
\end{align}
and
\[
\frac{\partial^2}{\partial \lambda^2_i} \frac{\pi_\mathrm{S}(\lambda)}{\pi_\mathrm{P}(\lambda)}
= \left( - \frac{d}{2} + 1 \right) \left( - \frac{d}{2} \right)
\left(\frac{ \lambda_1}{s_1} + \dotsb + \frac{\lambda_d}{s_d} \right)^{-\frac{d}{2}-1} \left( \frac{1}{s_i} \right)^2,
\]
we have
\begin{align}
\label{harmonicpoisson}
\Deltaxy{}{}
\frac{\pi_\mathrm{S}(\lambda)}{\pi_\mathrm{P}(\lambda)}
= \sum_i s_i \left( \lambda_i \frac{ \partial^2 }{\partial \lambda_i^2}
\frac{\pi_\mathrm{S}(\lambda)}{\pi_\mathrm{P}(\lambda)}
+ \frac{1}{2} \frac{\partial}{\partial \lambda_i}
\frac{\pi_\mathrm{S}(\lambda)}{\pi_\mathrm{P}(\lambda)}
\right) = 0.
\end{align}
Since $\pi_\mathrm{S}/\pi_\mathrm{P}$ is a non-constant positive superharmonic function of $\lambda$,
the Bayesian predictive density $\py_\mathrm{S}(y \mid x)$ based on $\pi_\mathrm{S}$
asymptotically dominates $\py_\mathrm{P}(y \mid x)$
by Corollary~\ref{usefulcor}.

The model manifold endowed with the predictive metric $\gxy{ij}{}$ is isometric to the first orthant
$\mathbb{R}^n_+ = \{(x^1, \dotsb, x^n): x^1 > 0, x^2 > 0 , \dotsb, x^n > 0 \}$
of the Euclidean space $\mathbb{R}^n$, as we see below.
Define
\[
 \xi^{i'}
= 2 \sqrt{\frac{\theta^{i'}}{s_{i'}}} ~~~ (i' = 1,\ldots,d).
\]
Then,
\begin{align*}
\frac{ \partial \theta^i}{\partial \xi^{i'}} & =
\begin{cases}
(s_{i'} \lambda_{i'} )^{\frac{1}{2}} & (i = i') \\
0 & (i \neq i').
\end{cases}
\end{align*}
Thus, from \eqref{poissonpm},
the coefficients of the metric with respect to $(\xi^{i'})$ are given by
\begin{align*}
\gxy{i' j'}{}&= \sum_{i,j} \frac{\partial \theta^i}{\partial \xi^{i'}} \gxy{ij}{} \frac{\partial \theta^j}{\partial \xi^{j'}} =
\begin{cases} 1 & (i' = j') \\ 0 & (i' \neq j').
\end{cases}
\end{align*}
This coincides with the usual metric on $\mathbb{R}^n_+$.

Here, the function
\[
\| \xi \|^{-d+2} \propto
\frac{\pi_\mathrm{S}(\lambda)}{\pi_\mathrm{P}(\lambda)}
= \left(  \frac{\lambda_1}{s_1} + \dotsb + \frac{ \lambda_d}{s_d} \right)^{- \frac{d}{2}+1}
\]
of $\xi$
is the Green function of the heat equation on $\mathbb{R}^n$
and plays an essential role in Bayesian methods for model manifolds isometric to the Euclidean space.
For example, the prior density $\| \mu \|^{-d+2} $ for the  $d$-dimensional
Normal model $N_d ( \mu, I_d)$, where $\mu$ is the $d$-dimensional unknown mean vector
and $I_d$ is the $ d \times d $ identity matrix, is known as the Stein prior.

The Bayesian predictive density based on $\piP$ is
\begin{align*}
\py_\mathrm{P}(y \mid \ x) =&
\frac{\displaystyle \int \prod_{i=1}^d \left\{\frac{{\lambda_i}^{x_i}}{x_i!} \ee^{-\lambda_i}
\frac{(s_i \lambda_i)^{y_i}}{y_i!} \ee^{-s_i \lambda_i} \right\}
\prod_{j=1}^d {\lambda_j}^{-1/2} \dd \lambda}
{\displaystyle \int \prod_{i=1}^d \left(\frac{{\lambda_i}^{x_i}}{x_i!} \ee^{-\lambda_i}\right)
\prod_{j=1}^d {\lambda_j}^{-1/2} \dd \lambda} \\
=&
\frac{\displaystyle \prod_{i=1}^d
\left\{\frac{s_i^{y_i}}{(1+s_i)^{x_i+y_i+1/2}}
\frac{\Gamma(x_i + y_i + 1/2)}{x_i! y_i!}\right\}}
{\displaystyle \prod_{i=1}^d \frac{\Gamma(x_i + 1/2)}{x_i!}} \\
=&
\displaystyle \prod_{i=1}^d
\left\{\frac{s_i^{y_i}}{(1+s_i)^{x_i+y_i+1/2}}
\frac{\Gamma(x_i + y_i + 1/2)}{\Gamma(x_i + 1/2) y_i!}\right\},
\end{align*}
where $\dd \lambda := \dd \lambda_1 \cdots \dd \lambda_d$.

The Bayesian predictive density based on $\pi_\mathrm{S}$ is
\begin{align*}
\py_\mathrm{S}(y \mid  x) &=
\frac{\displaystyle \int \prod_{i=1}^d \left\{ \frac{{\lambda_i}^{x_i}}{x_i!} \ee^{-\lambda_i}
\frac{(s_i \lambda_i)^{y_i}}{y_i!} \ee^{- s_i \lambda_i} \right\}
\biggl(\sum_{j=1}^d \frac{\lambda_j}{s_j} \biggr)^{-(d/2-1)}
\prod_{k=1}^d {\lambda_k}^{-1/2} \dd \lambda}
{\displaystyle \int \prod_{i=1}^d
\left(\frac{{\lambda_i}^{x_i}}{x_i!} \ee^{- \lambda_i} \right)
\biggl(\sum_{j=1}^d \frac{\lambda_j}{s_j} \biggr)^{-(d/2-1)}
\prod_{k=1}^d {\lambda_k}^{-1/2} \dd \lambda}\\
&=
\frac{\displaystyle \int \prod_{i=1}^d \left\{
 \frac{s_i^{y_i} \lambda_i^{x_i+y_i-1/2}}{y_i!} \ee^{- (1+s_i) \lambda_i} \right\}
 \left\{\int_0^\infty u^{\frac{d}{2}-2}
 \exp \Bigl( - u \sum_j \frac{\lambda_j}{s_j} \Bigr) \dd u \right\}
 \dd \lambda}
{\displaystyle \int \prod_{i=1}^d
 \left({\lambda_i}^{x_i-1/2} \ee^{-\lambda_i} \right)
 \left\{\int_0^\infty u^{\frac{d}{2}-2}
 \exp \Bigl( - u \sum_j \frac{\lambda_j}{s_j} \Bigr) \dd u \right\}
 \dd \lambda} \\
&=
\frac{\displaystyle
\int_0^\infty u^{\frac{d}{2}-2}
 \prod_{i=1}^d
 (1+s_i+u/s_i)^{-(x_i+y_i+1/2)}
 \dd u}
{\displaystyle
 \int_0^\infty u^{\frac{d}{2}-2}
 \prod_{i=1}^d
 (1+u/s_i)^{-(x_i+1/2)}
 \dd u}
\prod_{i=1}^d \frac{s_i^{y_i} \Gamma(x_i+y_i+1/2)}{y_i! \Gamma(x_i+1/2)}.
\end{align*}

We have the asymptotic risk difference
\begin{align}
N^2 \biggl[ \E \big\{ & D (\py(y \mid \theta), \py_{\mathrm S}(y \mid x ) \, \big| \, \theta \big\}
- \E \big\{ D (\py(y \mid \theta), \py_{\mathrm P}(y \mid x )  \, \big| \, \theta \big\} \biggr] \notag\\
= & \frac{\Deltaxy (\pi_\mathrm{S}/\pi_\mathrm{P})}{\pi_\mathrm{S}/\pi_\mathrm{P}}
- \frac{1}{2} \sum_{i,j} \gxy{}{ij} \frac{ \partial_i (\pi_\mathrm{S}/\pi_\mathrm{P})
\partial_j (\pi_\mathrm{S}/\pi_\mathrm{P})}{(\pi_\mathrm{S}/\pi_\mathrm{P})^2} + \mathrm{o}(1) \notag\\
=& - \frac{1}{2}  \left( \frac{d}{2} -1\right)^2
\left( \frac{\lambda_1}{s_1} + \dotsb + \frac{ \lambda_d}{s_d} \right)^{-1} + \mathrm{o}(1)
\label{poissonard}
\end{align}
by Theorem \ref{maintheorem}, \eqref{1stderiv}, \eqref{harmonicpoisson}, and
\begin{equation*}
\gxy{}{ij} =
\begin{cases}  s_i \lambda_i & ( i=j) \\ 0 & ( i \neq j).
\end{cases}
\end{equation*}
The asymptotic risk difference \eqref{poissonard} depends on $\lambda$ only through $\lambda_1/s_1 + \dotsb + \lambda_d/s_d$.
When $\lambda_1/s_1 + \dotsb + \lambda_d/s_d$ is small the improvement is large,
and it converges to zero as $\lambda_1/s_1 + \dotsb + \lambda_d/s_d$ goes to infinity.

It can be shown that $\pi_\mathrm{S}$ dominates $\piP$ in the sense of
infinitesimal prediction, and
we can construct
a Bayesian predictive density
dominating $\tilde{p}_\mathrm{P}(y \mid x)$
for arbitrary $s_i > 0$ $(i=1,\ldots,d)$
by modifying the prior $\pi_\mathrm{S}$.
Finite sample properties of this prior will be discussed in a another paper
by using an approach different from the asymptotic methods in the present paper.

\vspace{0.3cm}

In Examples 1, 2, and 3,
the volume element based on the predictive metric $\gxy{ij}{}$
coincides with the Jeffreys priors based on $g_{ij}$ and $\tilde{g}_{ij}$,
i.e.
$|\gxy{ij}{}(\theta)|^{1/2} \propto |g_{ij}(\theta)|^{1/2} \propto |\tilde{g}_{ij}(\theta)|^{1/2}$,
although the three metrics are different.
In general, if two metrics $g_{ij}$ and $\tilde{g}_{ij}$ satisfy the relation
\begin{align}
\label{prop}
 \tilde{g}_{ij} (\theta) = \sum_{k,l}  g_{kl} (\theta)  A^k_i  A^l_j,
\end{align}
where $(A^i_j)$ is a $d \times d$ regular matrix not depending on $\theta$,
then
\[
|\tilde{g}_{ij}|^{\frac{1}{2}} = |A^l_k| |g_{ij} |^{\frac{1}{2}}, ~~ \mbox{and} ~~
|\gxy{ij}{}|^{\frac{1}{2}} = |g_{ij}| |\tilde{g}_{ij}|^{-\frac{1}{2}} = |A^l_k|^{-1} |g_{ij} |^{\frac{1}{2}}
\]
and the volume elements based on $g_{ij}$, $\tilde{g}_{ij}$, and $\gxy{ij}{}$ are proportional to each other.
The relation \eqref{prop} appears in many examples as in Examples 1, 2, and 3.

\appendix

\section*{Appendix. Proofs of Theorems \ref{riskdiff} and \ref{maintheorem}}
    \setcounter{thm}{0}
    \renewcommand{\thethm}{\Alph{section}\arabic{thm}}

First, we prepare a preliminary result, Theorem \ref{risk-invariant}, to prove Theorem \ref{riskdiff}.

Asymptotic properties of predictive densities in the conventional setting in which $x(i)$, $i = 1,\ldots,N$, and $y$ have
the same distribution have been studied, see \citet{K96biometrika}, \citet{H98as}, and \citet{SDG06as}.

\citet{FKA04sjs} generalized these results for the setting in which $x(i)$, $i= 1, \ldots, N$, and $y$
have different distributions.
The Bayesian predictive density is expanded as
\begin{align}
\py_\pi&(y \mid x^N) = \py(y \mid \hat{\theta}_{\mathrm{mle}})
+ \frac{1}{2N} \sum_{i,j} \gx{}{ij}(\htheta_{\mathrm{mle}})
\left\{ \partial_i \partial_j \py(y \mid \htheta_{\mathrm{mle}}) - \sum_k \mGammay{ij}{k} \partial_k \py(y \mid \htheta_{\mathrm{mle}}) \right\} \notag \\
& +  \frac{1}{2N} \sum_k \left[ \sum_{i,j}  \gx{}{ij}(\htheta_{\mathrm{mle}}) \left\{\mGammay{ij}{k}(\htheta_{\mathrm{mle}})
- \mGammax{ij}{k}(\htheta_{\mathrm{mle}}) \right\} \right.  \notag \\
& ~~~~~~ \left. + 2 \sum_i \gx{}{ik}(\htheta_{\mathrm{mle}})
\left\{ \partial_i \log \pi (\htheta_{\mathrm{mle}}) - \sum_j \eGammax{ij}{j}(\htheta_{\mathrm{mle}}) \right\} \right]
\partial_k \py(y \mid \htheta_{\mathrm{mle}})
+ \mathrm{o}_\mathrm{p}(N^{-1}),
\label{FKA-1}
\end{align}
where $\htheta_{\mathrm{mle}}$ is the maximum likelihood estimator, and $\partial_i := \partial/\partial \theta^i$.
The estimator\break minimizing the Bayes risk
$\int \E [D \{\py(y \mid \theta), \py_\pi(y \mid x) \} | \theta] \pi(\theta) \dd \theta$
is given by
\begin{equation}
 \hat{\theta}^i_\pi = \hat{\theta}^i_{\sf mle} + \frac{1}{N} w^i_\pi(\hat{\theta}_{\sf mle})
+ \mathrm{o}_\mathrm{p}(N^{-1}),
\label{FKA-2}
\end{equation}
where
\begin{equation}
w^i_\pi(\theta) := \sum_k \gx{}{ik}(\theta) \Bigl\{ \partial_k \log \pi(\theta) - \sum_j \eGammax{kj}{j}(\theta) \Bigr\}
+ \frac{1}{2} \sum_{k, l}  \gx{}{kl}(\theta) \Bigl\{ \mGammay{kl}{i}(\theta) - \mGammax{kl}{i}(\theta) \Bigr\},
\label{FKA-3}
\end{equation}
which is a covariant vector.

The expansion of the risk function of a Bayesian predictive density $\py_\pi(y \mid x^N)$ up to the order $N^{-2}$ is given in
Theorem \ref{risk-invariant} below.
The expansion is invariant in the sense that
each term is a scalar not depending on parametrization.
In Theorem \ref{risk-invariant}, we put
\begin{align*}
\vxe{ij}{} (x ; \theta) :=& \partial_i \partial_j \log p(x \mid \theta) + g_{ij}(\theta)
- \sum_k \eGammax{ij}{k}(\theta) \partial_k \log p(x \mid \theta), \\
\vym{ij}{} (y ; \theta) :=& \frac{1}{\py(y \mid \theta)} \Bigl\{
\partial_i \partial_j \py(y \mid \theta) - \sum_k \mGammay{ij}{k} \partial_k \py(y \mid \theta) \Bigr\}, \\
T^{ijk} :=& \sum_{l, m, n}  T_{lmn} \gx{}{il} \gx{}{jm} \gx{}{kn},
\mbox{~~and~~}
\vxe{}{ik} := \sum_{j, l}  \vxe{jl}{} \gx{}{ij} \gx{}{kl}.
\end{align*}
Here, $\vxe{ij}{}$ and $\vym{ij}{}$ are vectors orthogonal to the model manifolds
$\{p(x \mid \theta) \, | \, \theta \in \Theta)\}$ and $\{\py(y \mid \theta) \, | \, \theta \in \Theta)\}$, respectively.
These vectors are closely related to the curvature of the manifolds.

\vspace{0.3cm}

\begin{thm}
\label{risk-invariant}
The expected Kullback--Leibler divergence from the true density $\py(y \mid \theta)$
to the Bayesian predictive density $\py_\pi(y \mid x^N)$ based on a prior $\pi(\theta)$ is expanded as
\begin{align}
\notag
\E \Bigl\{ & D \bigl(\py(y \mid \theta), \py_\pi(y \mid x^N) \bigr) \, \Big| \, \theta \Bigr\} \\ \notag
= & \frac{1}{2N} \sum_{i, j}  \gy{ij}{} \gx{}{ij}
+ \frac{1}{2N^2} \sum_{i, j}  \gy{ij}{} \upi^i \upi^j
+ \frac{1}{N^2} \sum_{i, j, k}  \gy{ij}{} \gx{}{jk} \enablay{k}{} \upi^i \\ \notag
& + \frac{1}{2N^2} \sum_{i, j, k, l}  \E \left( \vxe{}{ik} \vxe{}{jl} \, \Big| \, \theta \right) \gx{kl}{} \gy{ij}{} \\
&  - \frac{1}{2 N^2} \sum_{i, j, k, l} \E \left( \vym{ij}{} \vym{kl}{} \, \Big| \, \theta \right) \gx{}{ij} \gx{}{kl}
- \frac{1}{3N^2} \sum_{i, j, k}  \Ty{ijk}{} \Tx{}{ijk} \notag \\
& + \frac{3}{4N^2} \sum_{i, j, k, l}  \Qy{ijkl}{} \gx{}{ij} \gx{}{kl}
- \frac{1}{N^2} \sum_{i, j, k, l}  \E \left( \partial_i \ly \partial_j \ly \, \vym{kl}{} \, \Big| \, \theta \right) \gx{}{ik} \gx{}{jl} \notag \\
& + \frac{1}{4N^2} \sum_{i, j, k, l}  \E \left( \vym{ij}{} \vym{kl}{} \, \Big| \, \theta \right) \gx{}{ik} \gx{}{jl} \notag \\
& + \frac{1}{4N^2} \sum_{i, j, k, l, m, n}
\gy{ij}{}  (\mGammay{kl}{i} - \mGammax{kl}{i}) (\mGammay{mn}{j} - \mGammax{mn}{j}) \gx{}{km} \gx{}{ln} \notag \\
& - \frac{1}{N^2} \sum_{i, j, k, l, m}  \Ty{ijk}{} (\mGammay{lm}{k} - \mGammax{lm}{k}) \gx{}{il} \gx{}{jm}
+ \oo(N^{-2}),
\label{thm1}
\end{align}
where
\begin{equation*}
u^i_\pi(\theta) := \sum_k \gx{}{ik}(\theta) \Bigl\{ \partial_k \log \pi(\theta) - \sum_j \eGammax{kj}{j}(\theta) \Bigr\}
+ \sum_{k, l}  \gx{}{kl}(\theta) \Bigl\{ \mGammay{kl}{i}(\theta) - \mGammax{kl}{i}(\theta) \Bigr\}.
\end{equation*}
\end{thm}

\begin{proof}[Outline of the Proof]
Expansions of the risk functions corresponding to \eqref{thm1} when the distributions of $x(i)$, $i=1,\ldots,N$, and $y$ are the same
are obtained by \citet{K96biometrika} for curved exponential families by using differential geometrical notions
and by \citet{H98as} for general models under rigorous regularity conditions.
\citet{FKA04sjs} obtained several related results when when the distributions of $x(i)$, $i=1,\ldots,N$, and $y$ are different.
The expansion \eqref{thm1} is shown by lengthy calculations parallel to those in \citet{K96biometrika} and \citet{H98as}
by using the results such as \eqref{FKA-1}, \eqref{FKA-2}, and \eqref{FKA-3} obtained by \citet{FKA04sjs}.
\end{proof}

\vspace{0.3cm}

The quantity $\sum_{i, j, k, l}  \E \left( \vxe{}{ik} \vxe{}{jl} \, \Big| \, \theta \right) \gx{kl}{}$
is the Efron curvature \citep{E75as} of the model manifold $\{p(x \mid \theta) \, | \, \theta \in \Theta\}$
at $\theta$, and $\sum_{i, j, k, l}  \E \left( \vym{ij}{} \vym{kl}{} \, \Big| \, \theta \right) \gx{}{ij} \gx{}{kl}$
is the mixture mean curvature discussed in \cite{K96biometrika}
of the model manifold $\{\py(y \mid \theta) \, | \, \theta \in \Theta\}$ at~$\theta$.

\vspace{0.3cm}

\begin{proof}[Proof of Theorem \ref{riskdiff}]
The desired result is obvious from Theorem \ref{risk-invariant} because
\eqref{thm1} has the form
\begin{align}
\E & \Bigl\{  D \bigl(\py(y \mid \theta), \py_\pi(y \mid x^N) \bigr) \, \Big| \, \theta \Bigr\} \notag \\
&= \frac{1}{2N} \sum_{i, j}  \gy{ij}{} \gx{}{ij}
+ \frac{1}{2N^2} \sum_{i, j} \gy{ij}{} \upi^i \upi^j
+ \frac{1}{N^2} \sum_{i, j, k}  \gy{ij}{} \gx{}{jk} \enablay{k}{} \upi^i \notag \\
&~~~ + \text{terms independent of $\pi$} + \oo(N^{-2}).\label{corproof}
\end{align}
\end{proof}
To derive Theorem \ref{riskdiff},
it is sufficient to show \eqref{corproof}.
Much less calculation is needed to verify \eqref{corproof} than to obtain all the explicit
terms in \eqref{thm1}.

\vspace{0.3cm}

\begin{proof}[Proof of Theorem \ref{maintheorem}]
Let $f(\theta) := \pi(\theta) / \piP(\theta)$.
Since
\begin{align}
\frac{1}{2} & \sum_i \sum_j \gxy{}{ij} \partial_i \log f \partial_j \log f + \Deltaxy \log f
= \frac{\displaystyle \Deltaxy f}{f}
- \frac{1}{2} \sum_i \sum_j \gxy{}{ij} \frac{\partial_i f \partial_j f}{f^2}
= 2 \frac{\Deltaxy f^\frac{1}{2}}{f^{\frac{1}{2}}},
\label{kakikae}
\end{align}
it is sufficient to show that the left-hand side of \eqref{main} is equal to
$(1/2) \gxy{}{ij} \partial_i \log f \partial_j \log f + \Deltaxy \log f$.

From \eqref{piP} and \eqref{4-0}, we have
\begin{align*}
\partial_i \log \pi & = \partial_i \log f + \partial_i \log \piP
= \partial_i \log f + \sum_{j, k} \oGammayx{ijk}{} \gxy{}{jk}.
\end{align*}
Let $r^j := \sum_{k, l}  \gx{}{kl} \left(\mGammay{kl}{j} - \mGammax{kl}{j} \right)$
and $s^i := \sum_{k, j} \gx{}{ik} \bigl(\oGammayx{kj}{j} - \eGammax{kj}{j} \bigr)$.
Then, from \eqref{FKA-3},
\begin{align}
u_{\pi}^i =& \sum_{i, k} \gx{}{ik} \left(\partial_k \log \pi - \sum_j \eGammax{kj}{j} \right) + r^i \notag \\
=& \sum_k \gx{}{ik} \left(\partial_k \log f + \sum_j \oGammayx{kj}{j} - \sum_j \eGammax{kj}{j} \right) + r^i \notag \\
=& \sum_k \gx{}{ik} \partial_k \log f + s^i + r^i.
\label{u=sr}
\end{align}
Thus, when $\pi = \piP$, $u_\mathrm{P}^i = s^i + r^i$.
From \eqref{6-2}, we have
\begin{align}
N^2 & \biggl( \E \big[ D (\py(y \mid \theta), \py_\pi(y \mid x^N) \big]
- \E \big[ D (\py(y \mid \theta), \py_{\mathrm{P}}(y \mid x^N) \big] \biggr) \notag \\
=& \frac{1}{2} \sum_{i,j} \gy{ij}{} \upi^i \upi^j
+ \sum_{i, j, k} \gy{ij}{} \gx{}{jk} \left(\partial_k u_{\pi}^i + \sum_l \eGammay{kl}{i} u_{\pi}^l \right)  \notag \\
& - \frac{1}{2} \sum_{i, j} \gy{ij}{} u_{\mathrm{P}}^i u_{\mathrm{P}}^j
- \sum_{i, j, k} \gy{ij}{} \gx{}{jk} \left( \partial_k u_{\mathrm{P}}^i + \sum_l \eGammay{kl}{i} u_{\mathrm{P}}^l \right) + \oo(1) \notag \\
\label{riskdiff-2}
=& \frac{1}{2} \sum_{i, j} \gy{ij}{} (\sum_k \gx{}{ik} \partial_k \log f + s^i + r^i)
( \sum_l \gx{}{jl} \partial_l \log f + s^j + r^j) \notag \\
& - \frac{1}{2} \sum_{i, j} \gy{ij}{} (s^i + r^i) (s^j + r^j) \notag \\
& + \sum_{i, j, k} \gy{ij}{} \gx{}{jk} \left\{ \sum_l \partial_k ( \gx{}{il} \partial_l \log f)
     + \sum_{l, m}  \eGammay{kl}{i} \gx{}{lm} \partial_m \log f \right\} + \oo(1) \notag \\
= & \frac{1}{2} \sum_{i, j} \gxy{}{ij} \partial_i \log f \partial_j \log f
+ \sum_{i, j, k} \gy{ij}{} \gx{}{ik} (\partial_k \log f) (s^j + r^j) \notag \\
& + \sum_{i, j, k, l} \gy{ij}{} \gx{}{ik} \partial_k (\gx{}{jl} \partial_l \log f)
+ \sum_{i, j, k, l, m} \gy{ij}{} \gx{}{jk} \eGammay{jk}{i} \gx{}{lm} \partial_m \log f + \oo(1).
\end{align}
Let $L_i := \partial_i \log f$.
From \eqref{kakikae}, it is sufficient to show that
\begin{align}
\label{target}
\sum_{i, j, k} \gy{ij}{} & \gx{}{ik} L_k ( s^j + r^j ) + \sum_{i, j, k, l}  \gy{ij}{} \gx{}{ik} \partial_k (\gx{}{jl} L_l)
+ \sum_{j, k, l, m} \gx{}{jk} \eGammay{klj}{} \gx{}{lm} L_m
\end{align}
is equal to
$\Deltaxy \log f = \sum_{i, j} \partial_i ( \gxy{}{ij} L_j ) + \sum_{i, j, k} \oGammayx{ij}{i} \gxy{}{jk} L_k$.
Since
\[
0 = \partial_i \delta_f^l = \partial_i (\sum_m \gx{}{lm} \gx{mn}{})
= \sum_m (\partial_i \gx{}{lm}) \gx{mn}{} + \sum_m \gx{}{lm} (\partial_i \gx{mn}{}),
\]
we have
\[
\partial_i \gx{}{lm} = -\sum_{j, k}  \gx{}{jl} \gx{}{km} (\partial_i \gx{jk}{}).
\]
Thus, from
\begin{align*}
\sum_{i, j} \partial_i ( \gxy{}{ij} L_j )
=& \sum_{i, j, k, l} \partial_i ( \gx{}{ik} \gx{}{jl} \gy{kl}{} L_j) \\
=& \sum_{i, j, k, l}(\partial_i \gx{}{ik}) \gx{}{jl} \gy{kl}{} L_j + \sum_{i, j, k, l} (\partial_i \gy{kl}{}) \gx{}{ik} \gx{}{jl} L_j
+ \sum_{i, j, k, l} \gx{}{ik} \gy{kl}{} \partial_i (\gx{}{jl} L_j),
\end{align*}
we have
\begin{align*}
\sum_{i, j, k, l} \gy{kl}{} \gx{}{ik} \partial_i (\gx{}{jl} L_j) =&
\sum_{i, j} \partial_i (\gxy{}{ij} L_j)
- \sum_{i, j, k, l}  (\partial_i \gx{}{ik}) \gx{}{jl} \gy{kl}{} L_j - \sum_{i, j, k, l} (\partial_i \gy{kl}{}) \gx{}{ik} \gx{}{jl} L_j \\
=& \sum_{i, j} \partial_i (\gxy{}{ij} L_j)
+ \sum_{i, j, m, n} (\partial_i \gx{mn}{}) \gx{}{im} \gxy{}{jn} L_j - \sum_{i, j, k, l}  (\partial_i \gy{kl}{}) \gx{}{ik} \gx{}{jl} L_j.
\end{align*}
Hence,
because of the duality \eqref{duality} of the e-connection and the m-connection,
\eqref{target} is equal to
\begin{align*}
\sum_{i, j, k, l} & \gy{ij}{}  \gx{}{ik} L_k \gx{}{jl} \bigl( \sum_m \oGammayx{lm}{m} - \sum_m \eGammax{lm}{m} \bigr)
+ \sum_{i, j, k, l, m} \gy{ij}{} \gx{}{ik} L_k \gx{}{lm} \bigl(\mGammay{lm}{j} - \mGammax{lm}{j} \bigr) \\
& + \sum_{i, j} \partial_i (\gxy{}{ij} L_j)
 + \sum_{i, j, k, l}(\partial_i \gx{jl}{}) \gx{}{il} \gxy{}{jk} L_k - \sum_{i, j, k, l} (\partial_i \gy{jl}{}) \gx{}{il} \gx{}{jk} L_k
+ \sum_{j, k, l, m} \gx{}{jk} \eGammay{klj}{} \gx{}{lm} L_m \\
=& \sum_{i, j} \partial_i (\gxy{}{ij} L_j) + \sum_{k, l, m}  \gxy{}{kl} \oGammayx{lm}{m} L_k \\
& -\sum_{i, j, k, l}  \left(\eGammax{ijl}{} + \mGammax{ilj}{} - \partial_i \gx{jl}{} \right) \gx{}{il} \gxy{}{jk} L_k
+ \sum_{i, j, k, l} \left(\eGammay{ijl}{} + \mGammay{ilj}{} - \partial_i \gy{jl}{} \right) \gx{}{il} \gx{}{jk} L_k \\
=& \Deltaxy \log f.\hspace*{308pt}\mbox{\qedhere}
\end{align*}
\end{proof}
%

%% Appendix %%
% \appendix
% \section{}\label{}

%% Supplement Material %%
% \begin{supplement}
% \sname{}\label{}
% \stitle{}
% \slink[url]{}
% \sdescription{}
% \end{supplement}

%% References %%

% Acknowledgements

\begin{acknowledgement}
The author appreciates constructive comments of the associate editor.
This research was\break partially supported
by Grant-in-Aid for Scientific Research (23650144, 26280005).
\end{acknowledgement}

\end{document}